\documentclass[12pt,reqno]{article}
\usepackage{amsmath,amsthm,amsfonts,amssymb,amscd}

\usepackage[latin1]{inputenc}

\usepackage{amssymb}
\theoremstyle{plain}
\newtheorem{Theorem}{Theorem}
\newtheorem{Corollary}{Corollary}

\newtheorem{Lemma}{Lemma}
\newtheorem{Proposition}{Proposition}

\theoremstyle{Definition}
\newtheorem{Definition}{Definition}
\theoremstyle{Remark}
\newtheorem{Remark}{Remark}

\numberwithin{equation}{section}

\def\fa{{\mathcal{F}}}
\def\O{{\mathcal{O}}}

\def\po{{\partial}}

\def\la{{\lambda}}

\def\al{{\alpha}}

\def\bc{{\mathbb{C}}}

\def\rank{\operatorname{{rank}}}

\def\dim{\operatorname{{dim}}}

\def\Aff{\operatorname{{Aff}}}

\def\Id{\operatorname{{Id}}}

\def\Diff{\operatorname{{Diff}}}
\def\sing{\operatorname{{sing}}}
\def\Sing{\operatorname{{sing}}}
\def\codim{\operatorname{{codim}}}

\def\deg{\operatorname{{deg}}}
\def\rank{\operatorname{{rank}}}

\def\leaderfill{\leaders\hbox to .8em{\hss .\hss}\hfill}
\def\_#1{{\lower 0.7ex\hbox{}}_{#1}}

\def\fa{{\mathcal{F}}}
\def\O{{\mathcal{O}}}

\def\po{{\partial}}

\def\la{{\lambda}}

\def\al{{\alpha}}

\def\bc{{\mathbb{C}}}

\def\dim{\operatorname{{dim}}}

\def\Aff{\operatorname{{Aff}}}

\def\Id{\operatorname{{Id}}}

\def\Diff{\operatorname{{Diff}}}
\def\sing{\operatorname{{sing}}}
\def\Sing{\operatorname{{sing}}}
\def\codim{\operatorname{{codim}}}

\def\deg{\operatorname{{deg}}}
\def\rank{\operatorname{{rank}}}

\begin{document}
\title{Transversely Lie  holomorphic  foliations  on projective spaces}
\author{Albet\~a Mafra and Bruno Sc\'ardua}

\date{}
{\tiny \maketitle }


\pagenumbering{arabic}

\begin{abstract}
We prove that a one-dimensional foliation with generic
singularities on a projective space, exhibiting a Lie group
transverse structure in the complement of some codimension one
algebraic subset is logarithmic, i.e., it is the intersection of
codimension one foliations given by closed one-forms with simple
poles. If there is only one singularity in a suitable affine
space, then the foliation is induced by a linear diagonal vector
field.
\end{abstract}

\section{Introduction} \label{Section:intro}

In the study of foliations it is very useful to consider the
transverse structure\footnote{MSC Classification: 57R30, 22E15,
22E60.}\footnote{Keywords: Foliation, Lie transverse structure,
fibration.}. Among the simplest transverse structures are Lie
group transverse structure, homogeneous transverse structure and
Riemannian transverse structure. In the present work we consider
foliations with a Lie group transverse structure. In few words,
this means that the foliation is given by an atlas of submersions
taking values on a given Lie group $G$ and with transition maps
given by restrictions of left-translations on the  group $G$,
called the {\em transverse model} for $\fa$. We shall refer to
such a foliation as a {\em $G$-foliation}. The theory of
$G$-foliations is a well-developed subject and follows the
original work of Blumenthal \cite{Blumenthal}. In codimension one,
the Riemann-Koebe uniformization theorem implies that any
holomorphic $G$-foliation is indeed {\it transversely additive}
({\it i.e.}, of transverse model the additive group of complex
numbers) and therefore given by a closed holomorphic one-form
(\cite{Scardua1}). Using this and some extension techniques from
Holomorphic Foliations  it is possible to make an accurate study
of codimension one holomorphic foliations on complex projective
spaces admitting a Lie group transverse structure on the
complement of some invariant algebraic subset (of codimension one
in the nontrivial case) (\cite{Scardua1}, \cite{Scardua5}). On the
other hand, as for the local situation, it is possible to prove
that a nondegenerate isolated singularity of a holomorphic
one-form in dimension two, admitting a $G$-transverse structure in
the complement of its set of local separatrices which is assumed
to be composed of only finitely many curves, is analytically
equivalent to its formal normal form as introduced in
\cite{Martinet-Ramis1} and \cite{Martinet-Ramis2}. Our current aim
is to motivate the study of the codimension $\ell \geq 2$ case
beginning with the more basic situation. More precisely, we study
one-dimensional holomorphic foliations on complex projective
spaces, with {\em generic singularities}, and admitting a Lie
group transverse structure in the complement of some codimension
one algebraic invariant subset.

In codimension one,  a very basic example of foliation with
generic singularities and with Lie group transverse structure in
the complement of some algebraic hypersurface is the class of
logarithmic foliations. A {\it logarithmic foliation\,} on a
projective manifold $V$ is one given by a closed rational one-form
$\eta$ with simple poles. A {\it Darboux foliation} on $V$ is a
logarithmic foliation given by a rational one-form $\eta$ as
follows: $\eta=\sum\limits_{j=1}^r \lambda_j \dfrac{df_j}{f_j}$,
\, where the $f_j$ are rational functions on $V$ and $\lambda_j
\in \bc\setminus \{0\}$ (see \cite{Scardua5}  and \cite{Omegar}
for more on logarithmic
 and Darboux foliations). It is well-known that  any logarithmic
one-form $\eta$ in $\bc {\bf P}^m$ is of Darboux type
(\cite{Scardua1}). We extend these notions in a natural way. By a
{\it logarithmic foliation} of dimension one on a manifold $V^m$
we mean a foliation which is given by a system of $m-1$ closed
meromorphic one-forms $\eta_1,...,\eta_{m-1}$, all with simple
poles and linearly independent in the complement of their sets of
poles. A {\it Darboux foliation} of dimension one is a logarithmic
foliation given by one-forms $\eta_j$ of Darboux type. An isolated
singularity of a holomorphic vector field $X$ is {\it without
resonances} if the eigenvalues of the linear part $DX$ at the
singular point, are linearly independent over $\mathbb Q$. Local
linearization of such a (nonresonant) singularity is granted if
this singularity is in the {\it Poincar\'e domain}, i.e., the
convex hull of its eigenvalues does not contain the origin. In the
other case, i.e., in the {\it Siegel domain} there are Diophantine
conditions that assure local linearization (\cite{Brjuno},
 \cite{Siegel}).

Our main result below is a first step in the comprehension of the
possible Lie group  and homogeneous transverse structures (see
\cite{Godbillon}) for codimension $\ell\geq 1$ holomorphic
foliations on projective spaces.
\begin{Theorem}
\label{Theorem:main} Let $\fa$ be a one-dimensional holomorphic
foliation with singularities on $\bc {\bf P}^m$. Assume that:
\begin{enumerate}
\item The singularities of $\fa$ are linearizable without
resonances.

\item There is an invariant codimension one algebraic subset
$\Lambda \subset \bc {\bf P}^m$  such that $\fa$ has a
$G$-transverse structure in $\bc {\bf P}^m \setminus \Lambda$.
\end{enumerate}

Then $\fa$ is a logarithmic foliation.

If  moreover $\sharp [\sing(\fa) \cap \bc ^m]=1$ for an affine
space $\bc ^m \subset \bc {\bf P}^m$ with $\bc {\bf P}^m\setminus
\bc^m$ in general position with respect to $\fa$, then $\fa$ is
linearizable, {\it i.e.}, $\fa\big|_{\bc^m}$ is induced by a
vector field $X$ that can be put in the linear diagonal form
$X=\sum\limits_{j=1}^m a_j x_j \frac{\partial}{\partial x_j}$.

\end{Theorem}

As   spolium of the proof of Theorem~\ref{Theorem:main} we obtain:

\begin{Theorem}
\label{Theorem:spolium} Let $\fa$ be a one-dimensional holomorphic
foliation with singularities on a connected complex manifold
$V^m$, of dimension $m\geq 2$, such that:
\begin{enumerate}
\item There is a codimension one analytic subset $\Lambda\subset
V^m$ such that $\fa\big|_{V\setminus \Lambda}$ has a
$G$-transverse structure.

\item There is a singular point $p \in \sing(\fa)$ which is
linearizable without resonances.

\end{enumerate}

Then:

\begin{itemize}

\item[{\rm(i)}] $\fa\big|_{V\setminus \Lambda}$ is given by $m-1$
closed holomorphic one-forms $\eta_1,...,\eta_{m-1}$.

\item[{\rm(ii)}] Necessarily $p\in \Lambda$ and
$\eta_1,...,\eta_{m-1}$ extend meromorphic with simple poles to a
neighborhood of $p$ and therefore  to each irreducible component
of $\Lambda\setminus\sing(\fa)$ that intersects $V$.

\item[{\rm(iii)}] If each singularity $q\in\sing(\fa)$ is
linearizable without resonance, then $\fa$ is a logarithmic
foliation indeed, each $\eta_j$ extends to a meromorphic one-form
in $V$, with simple poles and $\fa$ is given by the system
$\{\eta_1,...,\eta_{m-1}\}$
\end{itemize}
\end{Theorem}

\section{$G$-foliations}
\label{section:DarbouxLie} Let $\fa$ be a codimension $\ell$
foliation on a manifold $V$. Given a Lie group $G$ of dimension
$\ell$ we say that $\fa$ {\it admits a Lie group transverse
structure of model} $G$ or a {\it $G$-transverse structure} for
short, if there are an open cover $V=\bigcup\limits_{j\in J} U_j$
of $V$ such that on each open set we have defined a submersion
$f_j \colon U_j \to G$ such that the leaves of $\fa\big|_{U_j}$
are  levels of $f_j$ and on each nonempty intersection $U_i \cap
U_j\ne \emptyset$ we have $f_i=g_{ij} f_j$ for some locally
constant map $g_{ij} \colon U_i \cap U_j \to G$. In other words:
$\fa$ is defined by the submersions $f_j \colon U_j \to G$ which
on $U_i \cap U_j$ differ by  left translations  $f_i=L_{g_{ij}}
(f_j)$ for some locally constant $g_{ij}\in G$. In this case we
call $\fa$ a {\it $G$-foliation}. The characterization of
$G$-foliations is given by the following result:

\begin{Theorem}[Darboux-Lie, \cite{Godbillon}]
\label{Theorem:Darboux-Lie} Let $\fa$ be a codimension $\ell$
foliation on $V$ and $G$ a Lie group of dimension $\ell$. If $\fa$
admits a $G$-transverse structure then there are one-forms
$\Omega_1,...,\Omega_\ell$ in $V$ such that:
\begin{itemize}
\item[{\rm(a)}]  $\{\Omega_1,...,\Omega_\ell\}$ is a rank $\ell$
integrable system which defines $\fa$.

\item[{\rm (b)}] $d\Omega_k =\sum\limits_{i<j} c_{ij}^k \Omega_i
\wedge \Omega_j$ where $\{c_{ij}^k\}$ are the structure constants
of a Lie algebra basis of $G$.
\end{itemize}

If $\fa, V$ and $G$ are complex {\rm(}holomorphic{\rm)} then the
$\Omega_j$ can be taken holomorphic. Given any basis
$\omega_1,...,\omega_\ell$ of the Lie algebra of $G$ the structure
constants $\{c_{ij}^k\}$ of this basis can be obtained above.

Conversely, given a maximal rank system of one-forms
$\Omega_1,...,\Omega_\ell$ in $V$ such that
$d\Omega_k=\sum\limits_{ij}^k c_{ij}^k\, \, \Omega_i \wedge
\Omega_j$, where the $\{c_{ij}^k\}$ are the structure constants of
the basis $\{\omega_1,...,\omega_\ell\}$ of the Lie algebra of
$G$, then:

\begin{itemize}

\item[{\rm (c)}] For each point $p\in V$ there is a neighborhood
$p\in U_p \subset V$ equipped with a submersion $f_p\colon U_p \to
G$ which defines $\fa$ in $U_p$ such that
$f_p^{\omega_j}=\Omega_j$ in $U_p$, for all $j\in \{1,...,q\}$.

\item[{\rm(d)}] If $V$ is simply-connected we can take $U_p = V$.

\item[{\rm(e)}]  If $U_p \cap U_q \ne \emptyset$ then in the
intersection we have $f_q = L_{g_{pq}}(f_p)$ for some locally
constant left translation $L_{g_{pq}}$ in $G$.

\end{itemize}
In particular, $\fa$ has a $G$-transverse structure.

\end{Theorem}

\section{Examples}

The most trivial example of a $G$-foliation is given by the
product foliation on a manifold $V=G\times N$ product of a Lie
group $G$ by a manifold $N$. The leaves of the foliation are of
the form $\{g\}\times N$ where $g\in G$. Other basic examples of
$G$-foliations are: \vglue.1in \noindent{\bf 1}.  Let $H$ be a
closed (normal) subgroup of a Lie group $G$. We consider the
action $\Phi\colon H\times G\rightarrow G$ given by
$\Phi(h,g)=h.g$ and the quotient map $\pi\colon G\rightarrow G/H$
(a fibration) which defines a foliation $\fa$ on $G$. Given
$x\in{\fa}_{g}=\pi^{-1}(Hg)$ we have $\pi(x)=Hg$ and
$\Phi_{h}(x)=h.x$. But $\pi(\Phi_{h}(x))=\pi(h.x)=H.hx=Hx$ implies
that $\Phi_{h}(x)\in \pi^{-1}(Hx)={\fa}_{x}$ and the orbit
$\O(g)=Hg$ is transverse to the fiber $\pi^{-1}(Hg)$. Hence, $\fa$
is a foliation invariant under the transverse action $\Phi$. Now
let $G$ be a simply-connected group, $H$ a discrete subgroup of
$G$ and $F\colon H\rightarrow \Diff(G)$ the natural representation
given by $F(h)=L_{h}$. The universal covering of $G/H$ is $G$ with
projection $\pi\colon G\rightarrow G/H$ and we have
$\pi_{1}(G/H)\simeq H$ because $\pi \circ f(g)=Hf(g)=Hg$ for $f\in
Aut(G)$, so $f(g)\simeq g$ implies that $f(g).g^{-1}\in H$ and
$f(g)=h.g$, for some unique $h\in H$. Therefore $f=L_{h}$ and then
we define $Aut(G)\rightarrow H;\ f\mapsto h$, which is an
isomorphism. So, we may write $F\colon \pi_{1}(G/H)\rightarrow
\Diff(G)$ and $\Phi\colon \pi_{1}(G/H)\times G\rightarrow G$. The
map $\Psi\colon H\times G\times G\rightarrow G\times G$ given by
$\Psi(h,g_{1},g_{2})=(L_{h}(g_{1}),L_{h}(g_{2}))$ is a properly
discontinuous action and defines a quotient manifold $V={G\times
G\over \Psi}$, which equivalence classes are the
orbits of $\Psi$. We have the following facts:\\
1) There exists a fibration $\sigma\colon V\rightarrow G/H$ with
fiber $G$, induced by $\pi\colon G\rightarrow G/H$, and structural
group isomorphic to $F(H)<\Diff(G)$.\\
2) The natural foliation $\fa$ on $G$ given by classes $Hg;\ g\in
G,$ is $\Phi$-invariant, such that the product foliation $G\times
\fa$ on $G\times G$ is $\Psi$-invariant and induces a foliation
${\fa}_{0}$ on $V$, called suspension of $\fa$ for $\Phi$,
transverse to $\sigma\colon V\rightarrow G/H$.

\vglue.1in \noindent{\bf 2}.  Let $G={\mathbb P}SL(2,{\mathbb C})$
and $H={\Aff}({\mathbb C})\triangleleft G$. An element of $G$ has
the expression $x\mapsto {ax+b\over cx+d}={a+{b\over x}\over
c+{d\over x}}$. The group $H$ is the isotropy group of $\infty$
has its maps given by $x\mapsto {ax+b\over d}\simeq
\left(\begin{array}{cc}
  a & b \\
  0 & {1\over a} \\
\end{array}\right)$. Since $H\triangleleft G,\ G$ has dimension 3 and $H$ has
dimension 2, we conclude that $G/H$ has dimension one. Thus we
have a fibration ${\mathbb P}SL(2,{\mathbb C})\rightarrow {\mathbb
C}{\bf P}^1\simeq S^{2}$ which is invariant under an action of
${\Aff}({\mathbb C})$ on ${\mathbb P}SL(2,{\mathbb C})$ having
leaves diffeomorphic to ${\mathbb C}^{*}\times {\mathbb C}$.

The  examples above are $G$-foliations which are also invariant
under a Lie group transverse action in the sense of
Definition~\ref{Definition:lieinvariant},
Section~\ref{section:lieinvariant}. As for the singular case, we
observe that a one-dimensional logarithmic foliation on a manifold
$V^m$ as defined in the introduction, is a $G$-foliation where $G$
is the additive Lie group in $\bc^{m-1}$, as granted by
Theorem~\ref{Theorem:Darboux-Lie},
Section~\ref{section:DarbouxLie}.

\section{$G$-Foliations with singularities}

Let now $\fa$ be a one-dimension holomorphic foliation with
isolated singularities on $V$. We shall say that $\fa$ admits a
$G$-transverse structure if the corresponding nonsingular
foliation $\fa_0:=\fa\big|_{V \setminus \sing(\fa)}$ admits a
$G$-transverse structure. Let $p\in \sing(\fa)$ be an isolated
singularity. Take a small neighborhood $p \in \Delta_p\subset V$
of $p$ such that $\Delta_p\cap \sing(\fa) =\{p\}$ and $\Delta_p$
is biholomorphic to a ball in $\bc^m$. Then $\Delta_p ^*=\Delta_p
\setminus \{p\}$ is simply-connected and by the classical theory
of transversely Lie foliations (cf. Darboux-Lie
Theorem~\ref{Theorem:Darboux-Lie}) the restriction
$\fa\big|_{\Delta_p ^*}$ is given by a holomorphic submersion $P
\colon \Delta _p^* \to G$. By Hartogs' extension theorem
(\cite{Gunning}, \cite{Gunning-Rossi}) the map extends to a
holomorphic map $P \colon \Delta _p \to G$, which is a {\it first
integral} for $\fa$ in $\Delta_p$. Assume now that $\fa$ admits a
$G$-transverse structure on $V\setminus \Lambda$ where
$\Lambda\cap \sing(\fa)=\emptyset$ and $\Lambda\subset V$ is a
codimension $\geq 2$ analytic subset. By classical Hartogs'
extension theorem the one-forms $\Omega_1,...,\Omega_\ell$
obtained in Theorem~\ref{Theorem:Darboux-Lie} extend
holomorphically to $V$ and then also by
Theorem~\ref{Theorem:Darboux-Lie} the $G$-transverse structure
extends to $V$. Hence we shall consider the case where $\fa$
admits a $G$-transverse structure on $V\setminus\Lambda$ and
$\Lambda\subset V$ is analytic of codimension one. Also we shall
assume that $\Lambda$ is invariant by $\fa$.

\begin{Lemma}
\label{Lemma:extension} Let $X=\sum\limits_{j=1}^m \lambda _j x_j
\frac{\partial}{\partial x_j}$ in a neighborhood $U$ of $0\in
\bc^m$, with $\{\lambda_1,...,\lambda_m\}$ linearly independent
over $\mathbb Q$. Denote by $\fa_X$ the one-dimensional foliation
induced by $X$. Suppose there are holomorphic one-forms
$\Omega_1,...,\Omega_{m-1}$ defined in $U_0:=U\setminus
[\bigcup\limits_{j=1}^{m}(x_j=0)]$ such that:

\begin{itemize}

\item[{\rm i.}] $d\Omega_k=\sum\limits_{i<j} c_{ij}^k \Omega_i
\wedge \Omega_j$ where $\{c_{ij}^k\}$ are the structure constants
of a Lie algebra of a Lie group $G$.

\item[{\rm ii.}] The foliation $\fa_X$ induced by $X$ is given in
$U_0$ by the integrable system $\{\Omega_1,...,\Omega_{m-1}\}$ of
maximal rank.

\end{itemize}
Then

\begin{enumerate}

\item $d\Omega_j=0, \forall j$.

\item $\Omega_j$ extends meromorphic to $U$ as a closed one-form
with simple poles.

\end{enumerate}

\end{Lemma}

\begin{proof}
For simplicity we assume $m=3$. Let $\Theta_1,\Theta_2$ be given
by $\Theta_j=\sum\limits_{j=1}^3 a_k^j \frac{dx_k}{x_k}, \, a_k ^j
\in \bc.$ Then $\Theta_j(X)=\sum\limits_{k=1}^3 a_{k}^j
\lambda_k$. Thus we can choose $\Theta_1, \Theta_2$ such that
these are linearly independent in the complement of
$\bigcup\limits_{j=1}^3 (x_j=0)$ and $\Theta_j(X)=0, j=1,2$.
Therefore $\fa_X$ is defined by the integrable system $\{\Theta_1,
\Theta_2\}$ in $U$. By this and (ii) we can write in $U_0$

\[
\Omega_1=a_1 \Theta_1 + a_2 \Theta_2, \, \, \Omega_2= b_1 \Theta_1
+ b_2 \Theta_2
\]
for some holomorphic functions $a_1, a_2, b_1, b_2$ in $U_0$ with
the property that $a_1 b_2 - b_1 a_2$ has no zero in $U_0$. Now,
since $\Theta_j$ is closed we have
\[
d\Omega_1= da_1 \wedge \Theta_1 + da_2 \wedge \Theta_2, \,
d\Omega_2 = db_1 \wedge \Theta_1 + db_2 \wedge \Theta_2
\]
Thus $d\Omega_1 \wedge \Theta_1 = (da_1 \wedge \Theta_1  + da_2
\wedge \Theta_2) \wedge \Theta_1 = da_2 \wedge \Theta_2 \wedge
\Theta_1$. From $d\Omega_1 = c_{12}^1 \Theta_1 \wedge \Theta_2$ we
obtain $d\Omega_1 = c_{12}^1 (a_1 b_2 - a_2 b_1) \Theta_1 \wedge
\Theta_2 \Theta _1=0$. Thus  we obtain
\[
da_2 \wedge \Theta_1 \wedge \Theta_2=0
\]

We have

\[
\Theta_1 \wedge \Theta_2 = (a_1 ^1 a_2 ^2 - a_1 ^2 a_ 2 ^1 )
\frac{dx_1 \wedge dx_2}{x_1 x_2} + (a_1 ^1 a_ 3 ^2 - a_1 ^2 a_ 3
^1) \frac{dx_1 \wedge dx_3}{x_1 x_3} + (a_ 2 ^1 a_ 3 ^2 - a_ 3 ^1
a _ 2 ^2) \frac{dx_2\wedge dx_3 }{x_2 x_3}
\]

Write $\Theta_1 \wedge \Theta_2 = \alpha_{12} \frac{dx_1 \wedge
dx_2}{x_1 x_2} + \alpha_{13}\frac{dx_1 \wedge dx_3}{x_1 x_3} +
\alpha_{23}\frac{dx_2 \wedge dx_3}{x_2 x_3}$. For a holomorphic
function $f(x_1,x_2,x_3)$ in $U_0$ we have
\[
df\wedge \Theta _ 1 \wedge \Theta_2 =\bigg[\frac{f_{x_1}
\alpha_{23}}{x_2 x_3} - \frac{f_{x_2} \alpha_{13}}{x_1 x_3} +
\frac{f_{x_3} \alpha_{12}}{x_1 x_2}\bigg] dx_1 \wedge dx_2 \wedge
dx_3
\]

Therefore $df \wedge \Theta_1 \wedge \Theta_2=0 \Leftrightarrow
\alpha_{23} x_1 f_{x_1} - \alpha_{13} x_2 f_{x_2} + \alpha_{12}
x_3 f_{x_3}=0$.

Now, if we write in Laurent series $f=\sum\limits_{i,j,k \in
\mathbb Z} f_{ijk} x_1 ^i x_2 ^j x_3 ^k$ then the last equation is
equivalent to
\[
\sum\limits_{i,j,k\in \mathbb Z} (i \alpha_{23} - j \alpha_{13} +
k \alpha_{12})f_{ijk} x_1 ^i x_2 ^j x_ 3 ^k =0
\]
which is equivalent to
\[
(i \alpha_{23} - j \alpha_{13} + k \alpha_{12})f_{ijk}=0, \, i,j,k
\in \mathbb Z
\]

Recall that

$\alpha_{12}=a_1 ^1 a_2 ^2 - a_1 ^2 a_ 2 ^1, \, \alpha_{13}=a_1 ^1
a_ 3 ^2 - a_1 ^2 a_ 3 ^1, \, \alpha_{23}=a_ 2 ^1 a_ 3 ^2 - a_ 3 ^1
a _ 2 ^2$ and that we have $a_1 ^1 \lambda_1 + a_ 2 ^1 \lambda_2 +
a_ 3 ^1 \lambda _ 3 =0, \, a_1 ^2 \lambda _1 + a_ 2 ^2 \lambda _2
+ a_ 3 ^2 \lambda _ 3 =0$, so that the complex vectors $\vec
\alpha:=(\alpha_{12}, \alpha_{13}, \alpha_{23})$ and $\vec
\lambda:=(\lambda_1, \lambda_2, \lambda_3)$ are collinear, i.e.,
there is $t\in\bc^*$ such that $ \alpha_{12}=t \lambda_1, \,
\alpha_{13}=t \lambda_2, \, \alpha_{23}=t \lambda_3$. Thus we have
$df\wedge \Theta_1 \wedge \Theta_2 = 0 \Leftrightarrow (i
\lambda_1 - j \lambda_2 + k \lambda_3)=0, \forall i,j,k\in \mathbb
Z$.

By the nonresonance hypothesis, the only solution to the last
equation is the trivial solution, therefore such function $f$ must
be constant. This implies  that $a_2$ is constant. Similarly $a_1,
b_1$ and $b_2$ are constant in $U_0$ and  we have
\[
\begin{pmatrix}
\Omega_1 \\ \Omega_2 \end{pmatrix} =C . \begin{pmatrix} \Theta_ 1
\\ \Theta_2 \end{pmatrix}
\]
for some nonsingular $2 \times 2$ complex matrix $C$. This proves
(1) and (2) in the lemma.
\end{proof}

\begin{Lemma}
\label{Lemma:simplepoles} Let $\fa$ be a one-dimensional
holomorphic foliation with singularities on $V^m$. Assume that:
\begin{enumerate}
\item The singularities of $\fa$ are linearizable without
resonances.

\item There is an invariant codimension one analytic subset
 $\Lambda \subset V^m$  such that $\fa$ has a
$G$-transverse structure in $V^m \setminus \Lambda$.
\end{enumerate}

Then $\fa$ is a logarithmic foliation on $V^m$.

\end{Lemma}

\begin{proof}
According to Darboux-Lie Theorem~\ref{Theorem:Darboux-Lie} there
is a system of holomorphic one-forms $\Omega_1,...,\Omega_{m-1}$
defined in $V_0:=V\setminus \Lambda$ such that:

\begin{itemize}

\item[{\rm i.}] $d\Omega_k=\sum\limits_{i<j} c_{ij}^k \Omega_i
\wedge \Omega_j$ where $\{c_{ij}^k\}$ are the structure constants
of a Lie algebra of the Lie group $G$.

\item[{\rm ii.}] The foliation $\fa$  is given in $V_0$ by the
integrable system $\{\Omega_1,...,\Omega_{m-1}\}$ of maximal rank.

\end{itemize}
Applying now Lemma~\ref{Lemma:extension} and denoting by $\eta_j$
the meromorphic extension of $\Omega_j$ to $V^m$, we conclude that
$\fa$ is given by a system of closed meromorphic one-forms
$\{\eta_1,...,\eta_{m-1}\}$ on $V^m$, such that each $\eta_j$ is
holomorphic on $V^m \setminus \Lambda$ and has simple poles
(contained on $\Lambda$). Therefore $\fa$ is a logarithmic
foliation.
\end{proof}

\begin{Remark}
{\rm We can  prove  a more general version of the above lemma as
follows: {\em Let $\fa$ be a one-dimensional foliation on $\bc
{\bf P}^m$.  Suppose that $\fa$ is given by a system of $m-1$
meromorphic  one-forms, $\eta_j$, such that each $\eta_j$ is
closed. Assume that each singularity
 of $\fa$ is linearizable without resonances. Then $\fa$ is
 logarithmic.}
\begin{proof}
Let $\Lambda \subset \bc {\bf P}^m$ be an irreducible component of
the polar set $(\eta_j)_\infty$ of $\eta_j$. We claim that
$\Lambda \cap \sing(\fa_X) \ne \emptyset$. Indeed, since $\Lambda$
is an irreducible component of $(\eta_j)_\infty$  and $\eta_j$ is
closed it follows that $\Lambda$ is invariant by the codimension
one foliation induced by $\eta_j$ and therefore $\Lambda$ is
invariant by $\fa$. The Index Theorem in \cite{Soares} and its
natural generalizations in \cite{Suwa}, then imply that
necessarily $\Lambda \cap \sing(\fa)=\emptyset$. Given any
singular point $p \in \sing(\fa)\cap \Lambda$ by hypothesis $\fa$
is linearizable without resonance at $p$ and by the proof of
Lemma~\ref{Lemma:extension} we conclude that $\eta_j$ has simple
poles in a neighborhood of $p$. This implies that $\Lambda$ has
order one in $(\eta_j)_\infty$.\end{proof}

} \end{Remark}

 From now on we assume that $V^m =\bc {\bf P}^m$ and
$\fa$ is a one-dimensional foliation as in
Theorem~\ref{Theorem:main}. Choose an affine space $\bc^m \subset
\bc {\bf P}^m$ such that the (projective) hyperplane
$E_\infty:=\bc {\bf P}^m\setminus \bc ^m$ is {\it in general
position with respect to $\fa$}, what means the following:

\begin{itemize}
\item $E_\infty\cap \sing(\fa)=\emptyset$.

\item $E_\infty$ is transverse to $\fa$ except for a finite number
of tangency points.

\item $E_\infty$ meets each irreducible component of $\Lambda$
transversely and at non-tangency points of $\fa$.

\end{itemize}

Given affine coordinates $(z_1,...,z_m)\in \bc^m = \bc {\bf P}^m
\setminus E_\infty$, choose irreducible polynomials
$f_1,...,f_{s}\in \bc[z_1,...,z_m]$ such that
$\Lambda=\bigcup\limits_{j=1}^s (f_j=0)$. Then, since each
$\eta_j$ is meromorphic on $\bc {\bf P}^m$, it is a rational
one-form. Moreover, $\eta_j$ is holomorphic on $\bc {\bf
P}^m\setminus \Lambda$ and has only simple poles. The following
result is well-known as the {\em Integration Lemma} (cf.
\cite{Scardua1} for instance):

\begin{Lemma}
\label{Lemma:integrationlemma} We can write $\eta_j\big|_{\bc ^m}
=\sum\limits_{k=1}^s a^j _k \frac{df_k}{f_k}$ for some complex
numbers $a^j _k \in \bc$.
\end{Lemma}

Using this writing we can prove:

\begin{Lemma}
\label{Lemma:transverseintersection} Given $j_1,...,j_r \in
\{1,...,s\}$ the hypersurfaces $\{f_{j_1}=0\}, ..., \{f_{j_r}=0\}$
intersect transversely.

\end{Lemma}
\begin{proof}
Indeed, given any point $p\in \{f_{j_1}=0\}\cap  ... \cap
\{f_{j_r}=0\}$ of non transverse intersection, we have
$df_{j_1}(p)\wedge ...\wedge df_{j_r}(p)=0$ and then the
codimension-one foliations $\fa_{\eta_j}$, defined by the
one-forms $\eta_j, \, j \in \{j_1,...,j_r\}$ are not transverse at
$p$. {\it A fortiori}, this implies that $\fa$ has a singularity
at $p$. However, as we have seen in the proof of
Lemma~\ref{Lemma:extension} the manifolds $\{f_j=0\}$ are the
separatrices of $\fa$ at $p$ and therefore they are transverse at
$p$.

\end{proof}

Given a singularity $p\in \sing(\fa)$, because the irreducible
components of $\Lambda$ through $p$ contain the local invariant
hypersurfaces of $\fa$ passing through $p$, we conclude that
$s\leq m$. We claim that  $s=m$. Indeed, thanks to the local
linearization, $\fa$ exhibits $m$ local invariant analytic
hypersurfaces through $p$, and each of these contained in some
irreducible component $\Lambda_j, \, j \in \{1,...,s\}$ of
$\Lambda$, as in the proof of Lemma~\ref{Lemma:extension}. What we
have to prove is that there exists no irreducible component
$\Lambda_j$ containing two different local invariant hypersurfaces
passing through $p$. This is a kind of no self-connection for $p$
in the real dynamics framework. In our case it will be a
consequence of the nonresonance hypothesis. Choose a polynomial
vector field $X$ with isolated singularities on $\bc^m$, such that
$\fa\big|_{\bc^m}$ is induced by $X$ and in particular $\sing(\fa)
= \sing(\fa)\cap \bc^m = \sing(X)$. Given a singularity $p\in
\sing(\fa)$, since  the irreducible components of $\Lambda$
through $p$ contain the local separatrices of $\fa$ through $p$,
we conclude that if some irreducible component $\Lambda_j\ni p $
contains two different local invariant hypersurfaces through $p$
then, because the residue of a closed meromorphic one-form is
constant along an irreducible component of its polar set, we have
two different local invariant hypersurfaces through $p$ for which
the residues of the forms $\eta_1,...,\eta_{m-1}$ are the same.
Since, as in the proof of Lemma~\ref{Lemma:extension},  there is a
linear relation involving the eigenvalues of $DX(p)$ and the
residues of the one-forms $\eta_j$ along the irreducible
components of $\Lambda$ through $p$, we conclude that the
eigenvalues of $DX(p)$ exhibit some resonance, contradiction. Thus
we have $s=m$ and also:

\begin{Lemma}
\label{Lemma:degreeone} If $\sharp [ \sing(\fa) \cap \bc ^m]=1$
then $f_1,...,f_m$ are degree one polynomials.

\end{Lemma}

\begin{proof}
In fact, since these polynomials are transverse and
$\sing(\fa)=\{f_{1}=0\} \cap  ... \cap  \{f_{m}=0\}\subset \bc^m$
we have $\sharp[\sing(\fa)\cap \bc ^m]= \prod\limits_{j=1}^m \deg
(f_j)$ so that $\deg(f_j)=1, \forall j=1,...,m$.

\end{proof}

By an affine change of coordinates we can assume that the
singularity $p=0\in \bc^m$ is the origin. The following lemma is
adapted from \cite{Ito-Scardua} and \cite{Bracci-Scardua}:
\begin{Lemma}
\label{Lemma:linearvectorfield} Let $X$ be a polynomial vector
field with isolated singularities on $\bc^m, m \geq 3$. Suppose
that $\eta_j(X)=0, j=1,...,m-1$ where $\{\eta_1,...,\eta_{m-1}\}$
is a system of closed meromorphic one-forms, with simple poles, of
maximal rank in the complement of $\bigcup\limits_{j=1}^{m-1}
(\eta_j)_\infty$ and that the corresponding one-dimensional
foliation  $\fa$ has a linearizable nonresonant singularity at
$p\in \bc^m$. Then $X$ is a  linear vector field in some affine
chart in $\bc^m$.

\end{Lemma}

\begin{proof}
 Let $\Lambda_1,...,\Lambda_m$ be the
irreducible components of $\Lambda \cap \bc^m$ and the irreducible
polynomials $f_1,\dots,f_m \in \bc[z_1,\dots,z_m]$ such that
$\Lambda_j = \{(z_1,\dots,z_m) \in \bc^m :
f_j(z_1,\dots,z_m)=0\}$, $j=1,\dots,m$.  The polynomials $f_j$
have degree one, vanish at the origin $0 = \Sing(X)$ of $\bc^m$
and  intersect transversely at every point. Thus, we can find an
affine change of coordinates $T(z_1,\dots,z_m) = (y_1,\dots,y_m)$
on $\bc^m$ such that $f_j(y_1,\dots,y_m) = y_j$\,, $j=1,\dots,m$.
The one-forms $\eta_j$ $(j=2,\dots,m)$ write as $\eta_j =
\sum\limits_{k=1}^m\al_k^j\cdot\frac{dy_k}{y_k}$ for some $\al_k^j
\in \bc$ and $X = \sum\limits_{j=1}^m A_j(y_1,\dots,y_m)\,
\frac{\po}{\po y_j}$ is polynomial such that $\eta_j(X) \equiv 0$,
$\forall\,j=2,\dots,n$. Since the hyperplanes $\{y_j=0\} =
\Lambda_j$ are $X$-invariant we must have $A_j(y_1,\dots,y_m) =
y_j\cdot B_j(y_1,\dots,y_m)$ for some polynomials $B_j \in
\bc[y_1,\dots,y_m]$.  Thus $X = \sum\limits_{j=1}^m y_j\cdot
B_j(y_1,\dots,y_m)\, \frac{\po}{\po y_j}\,\cdot$ Hence $0 =
\eta_j(X) = \sum\limits_{k=1}^m \al_k^j\cdot B_k(y_1.,\dots,y_m)$,
$\forall\, j \in \{2,\dots,m\}$.  Therefore $\sum\limits_{k=1}^m
\al_k^j\cdot B_k(0,\dots,0) = 0$, $\forall\, j \in \{2,\dots,m\}$.
Let then $Z := \sum\limits_{k=1}^m B_k(0,\dots,0)y_k\,
\frac{\po}{\po y_k}$ be a linear vector field. By construction we
have $\eta_j(Z)=0$, $\forall\,j \in \{2,\dots,n\}$. Therefore the
foliations $\fa(X)$ defined by $X$ and $\fa_Z$ defined by $Z$
coincide on $\bc^m$ and thus $X = H\cdot Z$ for some polynomial
$H$ on $\bc^m$. Since $\Sing(X) = \{0\} = \Sing(Z)$ the polynomial
$H$ has no zeros on $\bc^m$ and it is therefore constant say $H =
\la \in \bc^*$. Thus we have proved that $X$ is linear in the
variables $(y_1,\dots,y_m) \in \bc^m$ of the form
$$
X(y_1,\dots,y_m) = \la\cdot \sum_{k=1}^m B_k(0,\dots,0)y_k\,
\frac{\po}{\po y_k}\,\cdot
$$
This ends the proof.

\end{proof}

The proof of Theorem~\ref{Theorem:main} follows from the above
lemmas. Theorem~\ref{Theorem:spolium} follows from (the proof of)
Lemma~\ref{Lemma:simplepoles}.

\section{Foliations invariant under $G$-transverse actions}
\label{section:lieinvariant}

 Let $V$ be a manifold, $\fa$ a
codimension $\ell$ foliation on $V$ and $G$ a Lie group of
dimension $\dim G=\codim \fa=\ell$.

\begin{Definition}[\cite{Behague-Scardua}]
\label{Definition:lieinvariant} {\rm We  say that $\fa$ is  {\it
invariant under a transverse action} of the group $G$, $\fa$ is
$G$-i.u.t.a. for short, if there is an  action $\Phi\colon G\times
V\rightarrow V$ of $G$  on $V$ such that: (i) the action is {\it
transverse to $\fa$}, i.e., the orbits of this action have
dimension $\ell$ and intersect transversely the leaves of $\fa$
and (ii) $\Phi$ {\it leaves $\fa$ is  invariant}, i.e.,  the maps
$\Phi_{g}: x \mapsto \Phi(g,x)$ take leaves of $\fa$ onto leaves
of $\fa$.}
\end{Definition}

Let $\fa$ be a foliation on $V$ such that $\fa$ is $G$-i.u.t.a. It
is not difficult to prove the existence of a {\it Lie foliation
 structure  for $\fa$ on $V$ of model $G$} in the sense of
Ch. III, Sec. 2 of \cite{Godbillon}. We shall then say that $\fa$
{\it has $G$-transversal structure} and prove (with a
self-contained proof) the existence of a {\it development} for
$\fa$ as in Proposition 2.3, page 153 of \cite{Godbillon} (Ch.III,
Sec. 2).  A sort of strong form of this procedure can be found in
\cite{Behague-Scardua},  Section~4 with a self-contained proof
(Proposition~3).

In this section all foliations are assumed to be holomorphic,
though our arguments hold for  class $C^\infty$. First of all we
remark the existence of a Lie transverse structure for the
foliation invariant under a Lie group transverse action:

\begin{Proposition}\label{Proposition:Lieinvariant} Let $V^m$ be a manifold
equipped with a  codimension $\ell$ foliation $\fa$ invariant
under a transverse action of a Lie group $G$. Then $\fa$ has a Lie
transverse structure of model $G$. Indeed, there exists an
integrable system $\{\Omega_{1},...,\Omega_{\ell}\}$ of one-forms
defining $\fa$ on $V$ with
$d\Omega_{k}=\sum\limits_{i<j}c_{ij}^{k}\Omega_{i}\wedge
\Omega_{j}$, where $\{c_{ij}^{k}\}$ are the  structure constants
of the Lie algebra of $G$.
\end{Proposition}

\noindent{\bf Proof}. Let $\Phi$ a transverse action of $G$ on a
manifold $V$ of dimension $m$, $\fa$ a codimension $\ell$
foliation invariant under the Lie group transverse action $\Phi$,
$\{ X_{1},...,X_{\ell}\}$ a basis of Lie algebra of $G$ and $\{
\omega_{1},...,\omega_{\ell}\}$ the corresponding dual basis. We
have $d\omega_{k}=\sum\limits_{i<j}c_{ij}^{k}\omega_{i}\wedge
\omega_{j}$. We consider an open cover $\{ U_{i}\}$ of $V$ by
coordinate systems of $\fa$: given by local charts
$\phi_{i}=(\alpha_{i},\beta_{i})\colon U_{i}\subset V\rightarrow
D_{1}^{i}\times D_{2}^{i}\subset {\mathbb
C}^{m-\ell}\times{\mathbb C}^{\ell}$, and the projection
$\pi_{2}\colon{\mathbb C}^{m-\ell}\times G\rightarrow G$ on the
second factor. Given any point $x \in V$ we denote by $\fa_x$ the
leaf of $\fa$ through  $x$. Then each map $\pi_{2}\circ \phi_{i}$
defined by $\pi_{2}\circ \phi_{i}(x)=\beta_{i}(x)$ is a submersion
and the plaques of $\fa$ in $U_i$ are given by the intersections
${\fa}_{x}\cap U_{i}=\phi_{i}^{-1}(D_{1}^{i}\times \{
\beta_{i}(x)\})=(\pi_{2}\circ \phi_{i})^{-1}(\beta_{i}(x))$, that
is, the restriction ${\fa}|_{U_i}$ is given by $\pi_{2}\circ
\phi_{i}=$constant. Thus, if $U_{i}\cap U_{j}\not=\emptyset$ then
there exists a locally constant map $\gamma_{ij}\colon U_{i}\cap
U_{j}\rightarrow \{ \mbox{left translations on } G\}$ such that
$\pi_{2}\circ \phi_{i}(x)=\gamma_{ij}(x)(\pi_{2}\circ
\phi_{j}(x))$ for all $x\in U_{i}\cap U_{j}$. Now, given a point
$p\in V$, there exists a local chart $(U,\phi=(\alpha,\beta))$ of
$\fa$ with $\phi\colon U\rightarrow D_{1}\times D_{2}\subset
{\mathbb C}^{m-\ell}\times G,\ \phi(p)=(0,e)$ and $\phi$ takes
every leaf of ${\fa}|_{U}$ to a fiber ${\mathbb
C}^{m-\ell}\times\{$constant$\}$ of ${\mathbb C}^{m-\ell}\times
G$. We may write $\phi(\Phi_{g}(x))=(\xi(x,g),g)$, for all $(x,g)$
close to $(0,e)$.

 We define  one-forms in
$\Omega_1,...,\Omega_\ell$ in $V$ by setting in $U$ the definition
$\Omega_{1}=(\pi_{2}\circ
\phi)^{*}\omega_{1},...,\Omega_{\ell}=(\pi_{2}\circ
\phi)^{*}\omega_{\ell}$. Since $\omega_i$ is left-invariant (that
is, $L_{g}^{*}\omega_{i}=\omega_{i},\forall g\in G$) we conclude
that $\Omega_{i}$ is well defined. In $U$ we have
$\Omega_{k}=(\pi_{2}\circ \phi)^{*}\omega_{k}=(\pi_{2}\circ
\phi)^{*}d\omega_{k}$ and then
$d\Omega_{k}=(\pi_{2}\circ\phi)^{*}(\sum\limits_{i<j}c_{ij}^{k}\omega_{i}\wedge
\omega_{j})=\sum\limits_{i<j}c_{ij}^{k}(\pi_{2}\circ
\phi)^{*}\omega_{i}\wedge (\pi_{2}\circ
\phi)^{*}\omega_{j}=\sum\limits_{i<j}c_{ij}^{\nu}\Omega_{i}\wedge
\Omega_{j}$. This shows that $d\Omega_{\nu}\wedge
\Omega_{1}\wedge...\wedge \Omega_{\ell}=0$, because
$c_{ij}^{\nu}=c_{ji}^{\nu}$ and then, by Frobenius' Theorem
\cite{Camacho-LinsNeto}, the system $\{
\Omega_{1},...,\Omega_{\ell}\}$ defines an integrable field of
$\ell$-plans $\cal S$. The condition $T_{p}{\fa}\oplus
T_{p}\Phi_{p}(G)=T_{p}V$ implies that any vector, tangent to
$\fa$, has no projection over $G$. We have $d\pi_{2}(x,y)= \left(
\begin{array}{cc}
  0 & 0 \\
  0 & \Id \\
\end{array}
\right)\in M_{\ell\times m}({\mathbb C})$, so
$d\pi_{2}(x,y)(v,u)=u$, that is, $d\pi_{2}(x,y)|_{{\mathbb
C}^{m-\ell}\times \{0\}}\equiv 0$. For a tangent vector $v=(X,0)$
on ${\fa}_{p}$ we have $\Omega_{j}(p)(X,0)=(\pi_{2}\circ
\Phi)^{*}\omega_{j}(p)(X,0)=\omega_{j}(\pi_{2}\circ
\Phi(p))d(\pi_{2}\circ \Phi)(p)(X,0)=0$ because
$d(\Phi(p)(X,0)\in{\mathbb C}^{m-\ell}\times \{0\}$. Hence, every
one-form $\Omega_{j}$ is zero along the leaves of $\fa$ and, since
$\dim\ {\fa}= \dim\ {\cal S}$, we conclude that ${\fa}= {\cal S}$.
\qed

\begin{Remark}
\label{Remark:monatshefte} {\rm
Proposition~\ref{Proposition:Lieinvariant} above can be proved
using the fact that the connection tangent to a foliation is flat
and therefore its Cartan's curvature form is zero (see
\cite{Behague-Scardua}), nevertheless we have chosen to give
explicit constructive arguments which we think may be useful for
an eventual consideration of the singular case.}
\end{Remark}

Proposition~\ref{Proposition:Lieinvariant} and
Theorem~\ref{Theorem:main} then promptly give:

\begin{Corollary}
\label{Corollary:Lieinvariant} Let $\fa$ be a one-dimensional
holomorphic foliation with singularities on $\bc {\bf P}^m$.
Assume that:
\begin{enumerate}
\item The singularities of $\fa$ are linearizable without
resonances.

\item There is an invariant codimension one algebraic subset
$\Lambda \subset \bc {\bf P}^m$  such that $\fa$ is invariant
under a $G$-transverse action in $\bc {\bf P}^m \setminus
\Lambda$.
\end{enumerate}

Then $\fa$ is a logarithmic foliation. If  moreover $\sharp
[\sing(\fa) \cap \bc ^m]=1$ for an affine space $\bc ^m \subset
\bc {\bf P}^m$ such that $\bc {\bf P}^m\setminus \bc^m$ is in
general position with respect to $\fa$  then $\fa$ is
linearizable.

\end{Corollary}

\bibliographystyle{amsalpha}

\begin{tabular}{ll}
 Bruno Sc\'ardua\\
 Instituto de  Matem\'atica C.P. 68530\\
Universidade Federal do Rio de Janeiro\\
21.945-970 Rio de Janeiro-RJ\\
 BRAZIL\\
scardua@impa.br
\end{tabular}
\begin{tabular}{ll}
Albet\~{a} Costa Mafra\\
 Instituto de  Matem\'atica - \\
 Universidade Federal do Rio de Janeiro\\
  Caixa Postal 68530\\
CEP. 21.945-970 Rio de Janeiro-RJ\\
 BRAZIL
 e-mail: albetan@im.ufrj.br
\end{tabular}

\end{document}